\documentclass{amsart}

\usepackage{amssymb}
\usepackage{eucal}
\usepackage{color}
\usepackage{arydshln}
\usepackage{verbatim}

\usepackage{multirow,bigdelim}
\usepackage{amsfonts}
\usepackage{dsfont}
\usepackage{amsmath}
\usepackage{mathrsfs}
\usepackage{todonotes}
\usepackage{hyperref}
\usepackage{multimedia}
\usepackage{graphicx}
\usepackage{indentfirst}
\usepackage{bm}
\usepackage{amsthm}
\usepackage{mathrsfs}
\usepackage{latexsym}
\usepackage{amsmath}
\usepackage{amssymb}
\usepackage{microtype}
\usepackage{relsize}
\usepackage{hyperref}
\usepackage{tikz-cd}
\usepackage[backend=biber,style=numeric,autocite=superscript]{biblatex}
\allowdisplaybreaks[4]

\numberwithin{equation}{subsection}

\DeclareSymbolFont{SY}{U}{psy}{m}{n}
\DeclareMathSymbol{\emptyset}{\mathord}{SY}{'306}

\theoremstyle{plain}

\newtheorem{theorem}{Theorem}[section]

\newtheorem{lemma}[theorem]{Lemma}
\newtheorem{proposition}[theorem]{Proposition}
\theoremstyle{definition}
\newtheorem{definition}[theorem]{Definition}

\numberwithin{equation}{section}

\makeatletter
\@namedef{subjclassname@2020}{\textup{2020} Mathematics Subject Classification}
\makeatother
\begin{document}
	
	\title[Riccati Equations and the Third Problem of Halmos]{Riccati Equations and the Third Problem of Halmos}

	\author{ KUI JI and ZHENYUE WANG }
	
	\thanks{}
	

	\curraddr[ ]{School of Mathematical Sciences, Hebei Normal University,
		Shijiazhuang, Hebei 050024, China}

	\email[]{}

	\email[K.Ji]{jikui@hebtu.edu.cn} \email[Z.Y.Wang]{wangzhenyuemath@126.com}
	
	\begin{abstract}

	In \cite{fang2026invariantII}, Fang, Hou, Jiang, and Zhang
	reduced Halmos's third problem to invertible operators
	with conjugate bi-geometric form. An operator has such a
	form if, with respect to an orthogonal decomposition,
	it admits a representation
	\[
	T=
	\begin{pmatrix}
		A&C\\
		0&B
	\end{pmatrix},
	\qquad A^*,B\in B_1(O_\delta),
	\]
	where \(O_\delta=\{z\in\mathbb C:|z|<\delta\}\) for some
	\(\delta>0\), and \(B_1(O_\delta)\) denotes the rank-one
	Cowen--Douglas class.
	
	We restrict our attention to the case \(B=A^*\) with
	strictly positive coupling \(C\), and settle Halmos's
	third problem for this class of operators. More precisely,
	we prove that every such operator \(T\) is invertible
	and that \(T^{-1}\) has a nontrivial closed invariant
	subspace. This result contributes to the study of Halmos's third problem within the conjugate bi-geometric framework.
	The proof combines monotone solutions of operator
	Riccati equations with strong operator convergence
	to construct invariant subspaces for associated
	Hamiltonian operator matrices.

	\end{abstract}
	\maketitle
	

\section{Introduction}

Let \(\mathcal H\) be a complex separable infinite-dimensional
Hilbert space, and let \(\mathcal B(\mathcal H)\) denote the
algebra of bounded linear operators on \(\mathcal H\).
Throughout this paper, all invariant subspaces are understood
to be closed. An operator is called \emph{intransitive} if it
has an invariant subspace other than \(\{0\}\) and the whole
space. The invariant subspace problem asks whether every
operator in \(\mathcal B(\mathcal H)\) is intransitive.
Closely related to this question is the third of the ten
problems posed by Halmos in \cite{halmos1970ten}:

\medskip
\noindent
\textbf{Halmos's third problem.}
\emph{If an intransitive operator is invertible, is its inverse
	also intransitive?}
\medskip

The difficulty is that invariance under an invertible operator
does not generally imply invariance under its inverse.
Indeed, if \(\mathcal M\) is invariant under \(T\), then
\(T\mathcal M\subseteq\mathcal M\), whereas invariance of
\(\mathcal M\) under \(T^{-1}\) would require
\(T\mathcal M=\mathcal M\).
Thus, an invariant subspace of \(T\) need not itself provide
an invariant subspace of \(T^{-1}\).

A reduction theorem of Fang, Hou, Jiang, and Zhang
\cite{fang2026invariantII} provides a concrete framework
for studying this problem. They proved that Halmos's third
problem is equivalent to its restriction to invertible
operators of the form
\[
T=
\begin{pmatrix}
	A&C\\
	0&B
\end{pmatrix}
\in\mathcal B(\mathcal H_1\oplus\mathcal H_2),
\qquad
A^*,B\in B_1(O_\delta),
\]
where
\[
O_\delta=\{z\in\mathbb C:|z|<\delta\},
\qquad \delta>0,
\]
and \(B_1(O_\delta)\) denotes the rank-one Cowen--Douglas
class. Such operators are called \emph{conjugate bi-geometric
	operators}. The reduction retains the full difficulty of
Halmos's third problem while supplying analytic structure
through the diagonal entries. In particular, \(A\) is
injective with closed range of codimension one, and \(B\)
is surjective with a one-dimensional kernel. Invertibility of the full operator requires the coupling to connect these two defect spaces.

Within this framework, affirmative results have been
obtained under bundle-preserving conditions and suitable
power-growth assumptions; see
\cite{j2026geometric}.

In this paper, we restrict our attention to conjugate
bi-geometric operators with \(B=A^*\) and strictly positive
coupling \(C\). We settle Halmos's third problem for this
entire class of operators. More precisely, for every
\(A^*\in B_1(O_\delta)\) and every strictly positive \(C\),
the corresponding block operator \(T\) is invertible,
and its inverse \(T^{-1}\) has a nontrivial closed invariant
subspace. Thus, the case determined by the structural
conditions \(B=A^*\) and \(C>0\) admits a complete
affirmative solution.

In view of this reduction, our result contributes to the study of Halmos's third problem within the conjugate bi-geometric framework. Our main theorem is as follows.

\medskip
\noindent
\textbf{Main Theorem.}
Let
\[
T=
\begin{pmatrix}
	A&C\\
	0&A^*
\end{pmatrix}
\in\mathcal B(\mathcal H\oplus\mathcal H),
\]
where \(A^*\in B_1(O_\delta)\) for
	some \(\delta>0\). If \(C>0\), then \(T^{-1}\)
	is intransitive.
\medskip
Here \(C>0\) means that \(\langle Cx,x\rangle>0\)
for every nonzero \(x\in\mathcal H\);
no uniform positive lower bound is assumed.

The assumptions \(B=A^*\) and \(C>0\) are motivated
by the connection between the inverse operator and
Hamiltonian operator matrices.
For a unit vector \(e\in\ker A^*\), strict positivity gives
\(\langle Ce,e\rangle>0\), which ensures that \(T\)
is invertible.
The adjoint relation between the diagonal entries,
together with the self-adjointness of the coupling,
provides the required block symmetry in the inverse formula.
Positivity of the coupling then allows \(-iT^{-1}\)
to be represented, up to unitary equivalence, by a
Hamiltonian operator matrix with nonnegative off-diagonal
entries.
These sign conditions permit the use of global positive
Riccati solutions and their order properties.
The theorem requires neither compactness of the coupling
nor power-growth estimates on the diagonal entries.

The spectral properties of Hamiltonian operator matrices
have been studied from several perspectives.
Liu, Jin, and Eerdun \cite{LiuJinEerdu2024}
proved the symplectic self-adjointness of upper triangular
Hamiltonian operator matrices defined on diagonal domains
and obtained criteria for the emptiness of their residual
spectra.

Our approach connects the invariant subspace problem
with methods from differential equations, control theory,
and symplectic geometry.
Lawson and Lim \cite{LawsonLim2006SymplecticSemigroup}
developed a symplectic semigroup approach to operator
Riccati equations and established a global existence
theorem for positive solutions.
We combine their existence theorem with monotonicity
and strong operator convergence to construct invariant
subspaces from the evolution of graph subspaces.
We consider bounded Hamiltonian operators of the form
\[
H=
\begin{pmatrix}
	K&Q\\
	P&-K^*
\end{pmatrix},
\qquad
K\in\mathcal B(\mathcal H),\quad P,Q\ge0,
\]
and prove that each such operator has a nontrivial
invariant subspace. To this end, we study the
differential Riccati equation
\[
Y'(t)=K^*Y(t)+Y(t)K-Y(t)QY(t)+P,
\qquad Y(0)=0.
\]
Its solution exists for all \(t\ge0\), is positive,
and is nondecreasing in the operator order.
Moreover, the graph of \(-Y(t)\) evolves according to
\[
\operatorname{Graph}(-Y(t))
=e^{-tH}(\mathcal H\oplus\{0\}).
\]

The limiting argument does not require the Riccati
trajectory \(Y(t)\) to remain uniformly bounded.
Instead, we use the bounded transforms
\[
R_t=(I+Y(t))^{-1}.
\]
These operators form a decreasing family of positive
contractions and hence converge strongly to an operator
\(R\) satisfying \(0\le R\le I\).
We then show that
\[
\mathcal M=
\left\{
\begin{pmatrix}
	Rx\\
	-(I-R)x
\end{pmatrix}
:\ x\in\mathcal H
\right\}
\]
is a nontrivial closed \(H\)-invariant subspace.
This construction remains valid even when the Riccati
trajectory does not converge to a bounded solution of
the associated algebraic Riccati equation.

Finally, the explicit inverse formula for \(T\) shows
that \(-iT^{-1}\) is unitarily equivalent to a Hamiltonian
operator of the preceding form, with both off-diagonal
entries positive. The invariant subspace obtained
from the Riccati flow therefore yields a nontrivial
invariant subspace of \(T^{-1}\).
This establishes a connection between positivity of
the coupling in the conjugate bi-geometric model and
the order structure of operator Riccati equations.

The paper is organized as follows.
Section~2 recalls the operator-theoretic preliminaries.
Section~3 develops the Riccati construction of invariant
subspaces for Hamiltonian operators and applies it to
prove the main theorem.

\section{Preliminaries}

Throughout this paper, $\mathcal H$ denotes a complex separable
infinite-dimensional Hilbert space, and $\mathcal B(\mathcal H)$
denotes the algebra of all bounded linear operators on $\mathcal H$.
We assume that the inner product is linear in the first variable.

In \cite{article2}, M.~J.~Cowen and R.~G.~Douglas introduced a class of operators, denoted by $B_n(\Omega)$, whose eigenvalues with constant multiplicity $n$ are taken over a domain $\Omega\subset\mathbb{C}$.
The Cowen--Douglas class $B_n(\Omega)$ of operators of rank $n$ is defined as follows \cite{article2}:
\begin{enumerate}
	\item $\Omega\subset \sigma(T)=\{z\in\mathbb{C}\mid T-z \text{ is not invertible}\}$;
	\item $\operatorname{ran}(T-z)=\mathcal{H}$ for all $z\in\Omega$;
	\item $\bigvee_{z\in\Omega}\ker(T-z)=\mathcal{H}$;
	\item $\dim\ker(T-z)=n$ for all $z\in\Omega$.
\end{enumerate}

In particular, let
\[
O_\delta=\{z\in\mathbb C:|z|<\delta\},
\qquad \delta>0,
\]
and suppose that $A^*\in B_1(O_\delta)$. Since $0\in O_\delta$,
we have
\[
\operatorname{ran}A^*=\mathcal H,
\qquad
\dim\ker A^*=1.
\]
It follows that $A$ is injective and has closed range. If $e$ is
a unit vector spanning $\ker A^*$, then
\[
\operatorname{ran}A=(\ker A^*)^\perp=e^\perp.
\]
Consequently, $A$ is bounded below and $A^*A$ is invertible. The
operator
\[
A_{\mathrm L}^{-1}:=(A^*A)^{-1}A^*
\]
is a bounded left inverse of $A$, and
\[
A_{\mathrm L}^{-1}A=I,
\qquad
AA_{\mathrm L}^{-1}=I-e\otimes e,
\]
where
\[
(e\otimes e)x=\langle x,e\rangle e,
\qquad x\in\mathcal H.
\]

\begin{definition}\label{def:1.1}\cite{fang2026invariantII}
	Let $T\in\mathcal{B}(\mathcal{H})$. We say that $T$ has a \textit{conjugate bi-geometric form}, if there is a nontrivial subspace $M$ such that
	\[
	T=
	\begin{pmatrix}
		T_1 & T_{12}\\
		0 & T_2
	\end{pmatrix}\!\!\!\!
	\begin{matrix}
		M\\
		\ \ M^{\perp}
	\end{matrix}
	\]
	where $T_1^*, T_2\in B_1(O_\delta)$ and $O_\delta\triangleq\{z\in\mathbb{C};\,|z|<\delta\}$ for some $\delta>0$.
\end{definition}

\begin{definition}\cite{IJMJjiang2025cowen}
	Let \(A\in B_{1}(\Omega)\) with \(0\in \Omega\). Then there exists an orthonormal basis \(\{e_{n}\}_{n=0}^{\infty}\) of \(\mathcal{H}\) such that 
	\(e_{n-1}\in \ker A^{n}\ominus\ker A^{n-1}\) for \(n\ge 1\). If \(AB=I\) and 
	\[
	\langle Be_{j},e_{i}\rangle=\begin{cases}
		0, & i=0,\\[2mm]
		b_{i,\,j}, & i\ge 1,
	\end{cases}
	\qquad j\ge 0,
	\]
	then \(B\) is called the \textit{standard right inverse} of \(A\).
	Dually, let $A^*\in B_1(\Omega)$ with $0\in\Omega$.
	If $C^*$ is the standard right inverse of $A^*$,
	then $C$ is called the \textit{standard left inverse} of $A$.
\end{definition}

\begin{definition}\cite{LiuJinEerdu2024}
	An operator matrix of the form
\[
H=
\begin{pmatrix}
	K&Q\\
	P&-K^*
\end{pmatrix}
\in\mathcal B(\mathcal H\oplus\mathcal H),
\]
where $K\in\mathcal B(\mathcal H)$ and $P=P^*$, $Q=Q^*$,
is called a Hamiltonian operator matrix.
\end{definition}

If
\[
J=
\begin{pmatrix}
	0&-I\\
	I&0
\end{pmatrix},
\]
then $H$ satisfies
\[
H^*J+JH=0.
\]
In Section~3, we consider the differential Riccati equation
\[
Y'(t)
=
K^*Y(t)+Y(t)K-Y(t)QY(t)+P,
\qquad Y(0)=0,
\]
where $P,Q\ge0$. The connection between this equation and the
evolution generated by the Hamiltonian operator matrix $H$ will
be used to construct an invariant subspace.

\section{Main result}

We begin with a standard existence and uniqueness theorem
for linear differential equations in Banach spaces.

\begin{lemma}\cite{Driver2004AnalysisTools}\label{yinli2.01}
	Let \(X\) be a real or complex Banach space, let \(a>0\), and let
	\[
	\mathcal A\in C([0,a],\mathcal B(X)),
	\qquad f\in C([0,a],X),
	\]
	where the continuity of \(\mathcal A\) is understood in the operator norm. Then, for every \(x_0\in X\), the initial value problem
	\[
	\begin{cases}
		x'(t)=\mathcal A(t)x(t)+f(t),\\
		x(0)=x_0
	\end{cases}
	\]
	has a unique solution
	\[
	x\in C^1([0,a],X)
	\]
	on the whole interval \([0,a]\). Here the derivative is defined in the norm of \(X\), and at the endpoints one-sided derivatives are taken.

\end{lemma}

We shall also need the invertibility of the solution to an
operator-valued linear initial value problem.

\begin{lemma}\label{yinli6}
Let
\[
M\in C([0,\infty),\mathcal B(\mathcal H)),
\]
where the continuity is understood in the operator norm. Then the initial value problem
\[
U'(t)=-M(t)U(t),\qquad U(0)=I
\]
has a unique solution \(U(t)\) that is continuously differentiable in the operator norm. For every \(t\ge 0\), \(U(t)\) is invertible, and
\[
\bigl(U(t)^{-1}\bigr)'=U(t)^{-1}M(t).
\]
\end{lemma}
\begin{proof}
Let
\[
X=\mathcal B(\mathcal H).
\]
For each \(t\ge0\), define linear operators on \(X\) by
\[
\mathcal A(t)Z=-M(t)Z,\qquad
\mathcal C(t)Z=ZM(t),
\qquad Z\in X.
\]
Then
\[
\|\mathcal A(t)\|_{\mathcal B(X)}\le\|M(t)\|,
\qquad
\|\mathcal C(t)\|_{\mathcal B(X)}\le\|M(t)\|.
\]
Moreover,
\[
\begin{aligned}
	\|\mathcal A(t)-\mathcal A(s)\|_{\mathcal B(X)}
	&\le\|M(t)-M(s)\|,\\
	\|\mathcal C(t)-\mathcal C(s)\|_{\mathcal B(X)}
	&\le\|M(t)-M(s)\|.
\end{aligned}
\]
Thus, both \(\mathcal A(t)\) and \(\mathcal C(t)\) are continuous
in the operator norm.

Applying Lemma~\ref{yinli2.01} with \(f=0\) and \(x_0=I\),
we obtain, for every \(a>0\), unique solutions to the initial
value problems
\[
U'(t)=-M(t)U(t),\qquad U(0)=I,
\]
and
\[
V'(t)=V(t)M(t),\qquad V(0)=I,
\]
that are continuously differentiable in the operator norm
on \([0,a]\). By uniqueness, the solutions on different
finite intervals agree on their overlaps. They therefore
define unique solutions \(U(t)\) and \(V(t)\) on
\([0,\infty)\), both continuously differentiable in the
operator norm.

A direct calculation gives
\[
	(V(t)U(t))'
	=V'(t)U(t)+V(t)U'(t)
	=V(t)M(t)U(t)-V(t)M(t)U(t)=0.
\]
Since \(V(0)U(0)=I\), it follows that
\[
V(t)U(t)=I.
\]
Set
\[
E(t)=U(t)V(t)-I.
\]
Then
\[
\begin{aligned}
	E'(t)
	&=U'(t)V(t)+U(t)V'(t)\\
	&=-M(t)U(t)V(t)+U(t)V(t)M(t)\\
	&=-M(t)\bigl(E(t)+I\bigr)
	+\bigl(E(t)+I\bigr)M(t)\\
	&=-M(t)E(t)+E(t)M(t).
\end{aligned}
\]
Since \(E(0)=0\), integration yields
\[
E(t)=\int_0^t
\bigl(E(u)M(u)-M(u)E(u)\bigr)\,du.
\]
Fix an arbitrary finite interval \([0,a]\). By continuity,
\[
m=\sup_{0\le s\le a}\|M(s)\|<\infty,
\qquad
C=\sup_{0\le s\le a}\|E(s)\|<\infty.
\]
Hence
\[
\|E(t)\|\le 2m\int_0^t\|E(s)\|\,ds,
\qquad 0\le t\le a.
\]
Using \(\|E(s)\|\le C\), we first obtain
\[
\|E(t)\|\le C(2mt).
\]
Iterating this estimate gives
\[
\|E(t)\|
\le C\frac{(2mt)^n}{n!},
\qquad n=0,1,2,\dots.
\]
For each fixed \(t\), we have
\[
\frac{(2mt)^n}{n!}\longrightarrow0
\qquad\text{as }n\to\infty.
\]
Therefore, \(\|E(t)\|=0\), and hence \(E(t)=0\).
Consequently,
\[
U(t)V(t)=I,\qquad 0\le t\le a.
\]
Together with \(V(t)U(t)=I\), this shows that \(U(t)\)
is invertible and
\[
U(t)^{-1}=V(t).
\]
Since \(a>0\) was arbitrary, \(U(t)\) is invertible for
every \(t\ge0\), and its inverse satisfies
\[
\bigl(U(t)^{-1}\bigr)'=U(t)^{-1}M(t).
\]
	
\end{proof}

We next recall the global existence theorem of Lawson and Lim
for positive solutions of operator Riccati equations.

\begin{lemma}\cite{LawsonLim2006SymplecticSemigroup}\label{yinli8}
		Let $\mathcal H$ be a real Hilbert space, and let
	\[
	A,R,S:[t_0,\infty)\longrightarrow\mathcal B(\mathcal H)
	\]
	be norm-continuous operator-valued functions.
	Suppose that
	\[
	R(t)\ge 0,\qquad S(t)\ge 0
	\quad\text{for all }t\ge t_0,
	\]
	and that $K_0\in\mathcal B(\mathcal H)$ satisfies $K_0\ge 0$.
	Then the initial value problem
	\[
	\dot K(t)
	=
	R(t)+A(t)K(t)+K(t)A(t)^*-K(t)S(t)K(t),
	\qquad K(t_0)=K_0,
	\]
	admits a solution
	\[
	K\in C^1\bigl([t_0,\infty),\mathcal B(\mathcal H)\bigr)
	\]
	such that $K(t)\ge 0$ for every $t\ge t_0$.
	Here $A(t)^*$ denotes the adjoint on the real Hilbert space
	$\mathcal H$, and positivity means positive self-adjointness.
\end{lemma}

To apply Lemma~\ref{yinli8} in our setting, we establish its
complex Hilbert space version by realification.

\begin{lemma}\label{fude}
	Let \(\mathcal H\) be a complex Hilbert space, and let
	\[
	A,R,S:[t_0,\infty)\longrightarrow\mathcal B(\mathcal H)
	\]
	be continuous in the operator norm.
	Suppose that \(R(t),S(t)\ge0\) for all \(t\ge t_0\),
	and let \(K_0\in\mathcal B(\mathcal H)\) satisfy \(K_0\ge0\).
	Then the initial value problem
	\[
	Y'(t)=R(t)+A(t)Y(t)+Y(t)A(t)^*-Y(t)S(t)Y(t),
	\qquad Y(t_0)=K_0,
	\]
	admits a positive self-adjoint solution
	\[
	Y\in C^1\bigl([t_0,\infty),\mathcal B(\mathcal H)\bigr).
	\]
\end{lemma}

\begin{proof}
	Regard \(\mathcal H\) as a real vector space, denoted by
	\(\mathcal H_{\mathbb R}\), equipped with the real inner product
	\[
	\langle x,y\rangle_{\mathbb R}
	=\operatorname{Re}\langle x,y\rangle.
	\]
	A direct calculation gives
	\[
	\|x\|_{\mathcal H_{\mathbb R}}^2
	=\langle x,x\rangle_{\mathbb R}
	=\langle x,x\rangle
	=\|x\|_{\mathcal H}^2.
	\]
	Thus, \(\mathcal H_{\mathbb R}\) is a real Hilbert space.
	Moreover,
	\[
	\langle Ax,y\rangle_{\mathbb R}
	=\operatorname{Re}\langle Ax,y\rangle
	=\operatorname{Re}\langle x,A^*y\rangle
	=\langle x,A^*y\rangle_{\mathbb R}.
	\]
	Hence the complex adjoint \(A^*\) is also the real adjoint
	of \(A\) on \(\mathcal H_{\mathbb R}\).
	Furthermore,
	\[
	\langle Rx,x\rangle_{\mathbb R}
	=\operatorname{Re}\langle Rx,x\rangle
	=\langle Rx,x\rangle\ge0.
	\]
	Consequently, \(R(t),S(t)\), and \(K_0\) are also positive
	operators on \(\mathcal H_{\mathbb R}\).
	Applying Lemma~\ref{yinli8} on \(\mathcal H_{\mathbb R}\),
	we obtain a positive self-adjoint solution \(Y(t)\).
	
	Define
	\[
	J:\mathcal H_{\mathbb R}\longrightarrow\mathcal H_{\mathbb R},
	\qquad Jx=ix.
	\]
	To show that \(Y(t)\) is complex-linear, it suffices to prove that
	\[
	JY(t)=Y(t)J.
	\]
	Set
	\[
	Z(t)=Y(t)J-JY(t).
	\]
	We shall prove that \(Z(t)=0\).
	Since \(A,R,S\), and \(K_0\) are complex-linear, we have
	\[
	AJ=JA,\qquad
	A^*J=JA^*,\qquad
	RJ=JR,\qquad
	SJ=JS,\qquad
	K_0J=JK_0.
	\]
	Direct calculation yields
	\[
	Z'(t)=Y'(t)J-JY'(t),
	\]
	together with
	\[
	(AY)J-J(AY)
	=AYJ-AJY
	=A(YJ-JY)
	=AZ,
	\]
	\[
	(YA^*)J-J(YA^*)
	=YJA^*-JYA^*
	=(YJ-JY)A^*
	=ZA^*,
	\]
	and
	\[
	\begin{aligned}
		(YSY)J-J(YSY)
		&=YSYJ-YSJY+YSJY-JYSY\\
		&=YS(YJ-JY)+(YJ-JY)SY\\
		&=YSZ+ZSY,
	\end{aligned}
	\]
	where we have used \(SJ=JS\).
	
	It follows that
	\[
	Z'
	=AZ+ZA^*-YSZ-ZSY
	=(A-YS)Z+Z(A^*-SY),
	\]
	with initial condition
	\[
	Z(t_0)=K_0J-JK_0=0.
	\]
	Write
	\[
	C(t)=A(t)-Y(t)S(t),
	\qquad
	D(t)=A(t)^*-S(t)Y(t).
	\]
	Then \(C(t)\) and \(D(t)\) are continuous in the operator norm,
	and
	\[
	Z'(t)=C(t)Z(t)+Z(t)D(t),
	\qquad Z(t_0)=0.
	\]
	Fix a finite interval \([t_0,t_1]\) and set
	\[
	M=\sup_{t_0\le s\le t_1}
	\bigl(\|C(s)\|+\|D(s)\|\bigr)<\infty.
	\]
	We have
	\[
	Z(t)=\int_{t_0}^t
	\bigl(C(s)Z(s)+Z(s)D(s)\bigr)\,ds,
	\]
	and hence
	\[
	\|Z(t)\|
	\le M\int_{t_0}^t\|Z(s)\|\,ds.
	\]
	Define
	\[
	f(t)=\int_{t_0}^t\|Z(s)\|\,ds.
	\]
	Then
	\[
	f(t)\ge0,\qquad
	f(t_0)=0,\qquad
	f'(t)=\|Z(t)\|\le Mf(t).
	\]
	Therefore,
	\[
	\frac{d}{dt}\Bigl(e^{-M(t-t_0)}f(t)\Bigr)
	=e^{-M(t-t_0)}\bigl(f'(t)-Mf(t)\bigr)\le0.
	\]
	Since \(e^{-M(t-t_0)}f(t)\ge0\) and its value at \(t=t_0\)
	is zero, we conclude that
	\[
	f(t)=0,\qquad t\in[t_0,t_1].
	\]
	Thus,
	\[
	Z(t)=0,\qquad t_0\le t\le t_1.
	\]
	As \(t_1>t_0\) was arbitrary, it follows that
	\[
	Z(t)=0,\qquad t\ge t_0.
	\]
	Hence \(Y(t)\) is complex-linear.
	Since it is self-adjoint on \(\mathcal H_{\mathbb R}\),
	it is also self-adjoint on the complex Hilbert space
	\(\mathcal H\). Consequently,
	\[
	\langle Y(t)x,x\rangle
	=\langle Y(t)x,x\rangle_{\mathbb R}\ge0,
	\qquad x\in\mathcal H.
	\]
	Thus, \(Y(t)\) is a complex-linear positive operator.
\end{proof}

We now apply these results to construct an invariant subspace
for a Hamiltonian operator matrix.

\begin{proposition}\label{yinli7}
Let \(K,P,Q\in\mathcal B(\mathcal H)\), with \(P,Q\ge0\). Then
	\[
	H=\begin{pmatrix}K&Q\\ P&-K^*\end{pmatrix}
	\]
	has a nontrivial invariant subspace.
\end{proposition}

\begin{proof}

Consider the equation
\begin{equation}\label{eq:riccati}
	Y'(t)=K^*Y(t)+Y(t)K-Y(t)QY(t)+P,
	\qquad Y(0)=0.
\end{equation}

\paragraph{Step 1: Solving the Riccati equation}

Applying Lemma~\ref{yinli8} and Lemma~\ref{fude} with
\(A(t)=K^*\), \(R(t)=P\), \(S(t)=Q\), and \(K_0=0\),
we obtain a positive self-adjoint solution \(Y(t)\)
of \eqref{eq:riccati} on \([0,\infty)\).

\paragraph{Step 2: Prove that \(Y(t)\) is monotonically increasing.}

Set \(D(t)=Y'(t)\). Differentiating \eqref{eq:riccati} gives
\[
	D'(t)
	=K^*D(t)+D(t)K-D(t)QY(t)-Y(t)QD(t)
	=(K-QY(t))^*D(t)+D(t)(K-QY(t)).
\]
By \eqref{eq:riccati} and \(Y(0)=0\), we have
\[
D(0)=P\ge0.
\]
Set \(N(t)=K-QY(t)\). By Lemma~\ref{yinli6},
there exists an invertible solution \(E(t)\) satisfying
\[
E'(t)=-N(t)E(t),\qquad E(0)=I.
\]
A direct calculation yields
\[	\frac{d}{dt}\bigl(E(t)^*D(t)E(t)\bigr)
	=(E(t)^*)'D(t)E(t)
	+E(t)^*D'(t)E(t)
+E(t)^*D(t)E'(t).
\]
Since
\[
D'(t)=N(t)^*D(t)+D(t)N(t)
\]
and
\[
(E(t)^*)'=(E'(t))^*=-E(t)^*N(t)^*,
\]
it follows that
\[
\frac{d}{dt}\bigl(E(t)^*D(t)E(t)\bigr)=0.
\]
Thus, \(E(t)^*D(t)E(t)\) is constant. Using \(E(0)=I\)
and \(D(0)=P\), we obtain
\[
E(t)^*D(t)E(t)=P,
\]
and hence
\[
D(t)=(E(t)^{-1})^*PE(t)^{-1}\ge0.
\]
Therefore, for any \(t\ge s\ge0\),
\[
Y(t)-Y(s)=\int_s^t D(u)\,du\ge0.
\]
Equivalently,
\[
0\le Y(s)\le Y(t),\qquad t\ge s\ge0.
\]
Thus, \(Y(t)\) is nondecreasing in the operator order.

\paragraph{Step 3: The evolution of the graphs under the operator $e^{-tH}$.}

Write
\[
G_t=\operatorname{Graph}(-Y(t))
=\left\{
\begin{pmatrix}
	x\\ -Y(t)x
\end{pmatrix}
\ \middle|\ x\in\mathcal H
\right\}.\tag{3.2}
\]
Let \(G(t)\) be the invertible solution of
\[
G'(t)=(-K+QY(t))G(t),\qquad G(0)=I.
\]
A direct calculation gives
\[
	(-Y(t)G(t))'
	=-Y'(t)G(t)-Y(t)G'(t)
	=-K^*Y(t)G(t)-PG(t).
\]
Recall that
\[
H=
\begin{pmatrix}
	K&Q\\
	P&-K^*
\end{pmatrix}.
\]
We then have
\[
\frac{d}{dt}
\begin{pmatrix}
	G(t)\\ -Y(t)G(t)
\end{pmatrix}
=-H
\begin{pmatrix}
	G(t)\\ -Y(t)G(t)
\end{pmatrix},
\]
with initial value
\[\begin{pmatrix}
	I\\0
\end{pmatrix}.
\]
Consequently,
\[
\begin{pmatrix}
	G(t)\\ -Y(t)G(t)
\end{pmatrix}
=e^{-tH}
\begin{pmatrix}
	I\\0
\end{pmatrix}.
\]
Since \(G(t)\) is invertible, it follows that
\[
\left\{
\begin{pmatrix}
	x\\ -Y(t)x
\end{pmatrix}
\ \middle|\ x\in\mathcal H
\right\}
=e^{-tH}(\mathcal H\oplus\{0\}).
\]
Thus,
\[
G_t=e^{-tH}(\mathcal H\oplus\{0\}).
\]
In particular, for any \(s,t\ge0\),
\[
G_{t+s}
=e^{-(t+s)H}(\mathcal H\oplus\{0\})
=e^{-sH}G_t.\tag{3.3}\label{3.3}
\]
Hence the graphs \(G_t\) evolve under \(e^{-tH}\).

\paragraph{Step 4: Compute the limit of \((I+Y(t))^{-1}\) in the strong operator topology.}

Set
\[
R_t=(I+Y(t))^{-1}.
\]
Since \(Y(t)\ge0\), we have \(I+Y(t)\ge I\), and therefore
\[
0\le R_t\le I.
\]
Moreover, the monotonicity of \(Y(t)\) implies that
\(R_t\) is nonincreasing in the operator order.

For each \(x\in\mathcal H\), the scalar function
\(\langle R_tx,x\rangle\) is nonnegative and nonincreasing,
so its limit exists as \(t\to\infty\). For \(t\ge s\),
we have
\[
0\le R_s-R_t\le I.
\]
Hence
\[
\begin{aligned}
	\|(R_s-R_t)x\|^2
	&=\langle (R_s-R_t)^2x,x\rangle\\
	&\le\langle (R_s-R_t)x,x\rangle\\
	&=\langle R_sx,x\rangle-\langle R_tx,x\rangle.
\end{aligned}
\]
To justify the inequality, set \(D=R_s-R_t\).
Then \(0\le D\le I\), and
\[
\langle (D-D^2)x,x\rangle
=\langle D^{1/2}(I-D)D^{1/2}x,x\rangle\ge0.
\]
Thus, \(D^2\le D\).

Since the right-hand side tends to zero as \(s,t\to\infty\),
the family \(R_tx\) satisfies the Cauchy criterion as
\(t\to\infty\). Define
\[
Rx=\lim_{t\to\infty}R_tx.
\]
The bound \(\|R_t\|\le1\) implies that \(R\) is a bounded
linear operator with \(\|R\|\le1\), and
\[
R_t\xrightarrow{\mathrm{SOT}}R.
\]
Furthermore, passing to the limit in \(0\le R_t\le I\) gives
\[
0\le R\le I.
\]

\paragraph{Step 5: Construction of a nontrivial closed subspace.}

Define
\[
Vx=
\begin{pmatrix}
	Rx\\
	-(I-R)x
\end{pmatrix},
\qquad
W(u,v)=(I-R)u+Rv,
\qquad x,u,v\in\mathcal H,
\]
and set
\[
M=\operatorname{ran}V.
\]
We claim that
\[
\operatorname{ran}V=\ker W.
\]
Indeed,
\[
WVx
=W\bigl(Rx,-(I-R)x\bigr)
=(I-R)Rx-R(I-R)x=0,
\]
so \(\operatorname{ran}V\subseteq\ker W\).

Conversely, let \((u,v)\in\ker W\). Then
\[
(I-R)u+Rv=0,
\]
and hence
\[
Rv=-(I-R)u.
\]
It follows that
\[
u=R(u-v).
\]
Taking \(x=u-v\), we obtain \(Rx=u\) and
\[
-(I-R)x=-x+Rx=-(u-v)+u=v.
\]
Thus, \((u,v)=Vx\), proving that
\(\ker W\subseteq\operatorname{ran}V\). Therefore,
\[
\operatorname{ran}V=\ker W.
\]
Since \(W\) is bounded, \(\ker W\) is closed, and hence
\(M=\operatorname{ran}V\) is a closed subspace.

For any \(x\ne0\), the identity
\[
x=Rx+(I-R)x
\]
shows that \(Rx\) and \((I-R)x\) cannot both vanish.
Thus, \(Vx\ne0\), and consequently
\[
M\ne\{0\}.
\]
On the other hand, for any \(x\ne0\),
\[
W(x,x)=(I-R)x+Rx=x\ne0.
\]
Hence \((x,x)\notin\ker W\), so
\[
\ker W\ne\mathcal H\oplus\mathcal H.
\]
Equivalently,
\[
M\ne\mathcal H\oplus\mathcal H.
\]
Thus, \(M\) is a nontrivial closed subspace.

\paragraph{Step 6: Prove that \(HM\subset M\).}

Define
\[
V_tx=
\begin{pmatrix}
	R_tx\\
	-(I-R_t)x
\end{pmatrix},
\qquad
W_t(u,v)=(I-R_t)u+R_tv.
\]
Since \(R_t=(I+Y(t))^{-1}\), we have
\[
Y(t)R_t=I-R_t.
\]
As \(R_t\) is invertible,
\[
	\operatorname{ran}V_t
	=\left\{
	\begin{pmatrix}
		R_tx\\
		-Y(t)R_tx
	\end{pmatrix}
	\ \middle|\ x\in\mathcal H
	\right\}
	=\operatorname{Graph}(-Y(t))
	=G_t.
\]
The same calculation as in Step~5 gives
\[
\operatorname{ran}V_t=\ker W_t.
\]
Fix \(s\ge0\). Then, by \eqref{3.3}, we obtain
\[
e^{-sH}\operatorname{ran}V_t
=e^{-sH}G_t
=G_{t+s}
=\operatorname{ran}V_{t+s}
=\ker W_{t+s}.
\]
It follows that
\[
W_{t+s}e^{-sH}V_t=0.
\]
Since
\[
V_t\xrightarrow{\mathrm{SOT}}V,
\qquad
W_t\xrightarrow{\mathrm{SOT}}W,
\]
and
\[
\|W_{t+s}\|\le\sqrt2,
\]
we obtain, for every \(x\in\mathcal H\),
\[
\begin{aligned}
	&\bigl\|W_{t+s}e^{-sH}V_tx-We^{-sH}Vx\bigr\|\\
	&\quad\le
	\bigl\|W_{t+s}\bigl(e^{-sH}V_tx-e^{-sH}Vx\bigr)\bigr\|
	+\bigl\|(W_{t+s}-W)e^{-sH}Vx\bigr\|\\
	&\quad\le
	\sqrt2\,\bigl\|e^{-sH}V_tx-e^{-sH}Vx\bigr\|
	+\bigl\|(W_{t+s}-W)e^{-sH}Vx\bigr\|
	\longrightarrow0
\end{aligned}
\]
as \(t\to\infty\). Therefore,
\[
We^{-sH}V=0,
\]
which implies
\[
e^{-sH}V\mathcal H
\subseteq\ker W
=\operatorname{ran}V
=M.
\]
Since \(V\mathcal H=M\), we conclude that
\[
e^{-sH}M\subseteq M,\qquad s\ge0.
\]
For any \(z\in M\) and \(s>0\),
\[
\frac{e^{-sH}z-z}{s}\in M.
\]
Moreover,
\[
\lim_{s\to 0^{+}}\frac{e^{-sH}z-z}{s}=-Hz.
\]
Since \(M\) is closed, it follows that \(Hz\in M\).
Thus,
\[
HM\subseteq M.
\]
\end{proof}

We now deduce the main theorem from Proposition~\ref{yinli7}.

\begin{proof}[Proof of the Main Theorem]
	
Choose a unit vector \(e\in\ker A^*\), and set
\[
d=\langle Ce,e\rangle.
\]
Since \(C>0\), we have \(d>0\). Hence $T$ is invertible.
Define
\[
L=(A^*A)^{-1}A^*
\left(I-\frac{Ce\otimes e}{d}\right),
\qquad
P=\frac{e\otimes e}{d},
\qquad
Q=LCL^*.
\]
A direct calculation gives
\[
LA=I,\qquad
AL=I-\frac{Ce\otimes e}{d},\qquad
LCe=0.
\]
Set
\[
F=
\begin{pmatrix}
	L&-Q\\
	P&L^*
\end{pmatrix}.
\]
Using the identities above and their adjoints,
a direct multiplication gives
\[
TF=FT=I.
\]
Thus, 
$
T^{-1}=F.
$
Moreover, \(P,Q\ge0\).
Set
\[
S=
\begin{pmatrix}
	iI&0\\
	0&I
\end{pmatrix},
\]
and let \(K=-iL\). Then
\[
S^{-1}(-iT^{-1})S=
\begin{pmatrix}
	K&Q\\
	P&-K^*
\end{pmatrix}.
\]
By Proposition~\ref{yinli7}, the operator
\(S^{-1}(-iT^{-1})S\) is intransitive.
Hence \(T^{-1}\) is intransitive.
\end{proof}

	\section{Declaration of generative AI and AI-assisted technologies in the manuscript preparation process}

During the preparation of this work we used GPT-6 Astra to assist with literature parsing, mathematical text formatting and English‑language polishing. After using this tool, we reviewed and edited the content as needed and take full responsibility for the content of the published article.

\printbibliography 
\end{document}